\documentclass{article}
\usepackage{wooyoung}

\DeclareMathOperator{\sech}{sech}

\def\ba{\begin{align}}
\def\ea{\end{align}}

\def\({\left(}
\def\){\right)}
\def\[{\left[}
\def\]{\right]}

\def\lb{\left|}
\def\rb{\right|}

\def\N{\mathbb{N}}

\def\R{\mathbb{R}}
\def\Z{\mathbb{Z}}
\def\C{\mathbb{C}}

\def \E{\mathbf{E}}

\def \CC{{\mathbb C}}
\def \DD{{\Delta}}

\def \RR{{\mathbb R}}

\newcommand{\seg}[2]{\stackrel{\line(1,0){#1}}{#2}}
\providecommand{\keywords}[1]{\textbf{\textit{Keywords--}} #1}

\begin{document}

\title{{A note on invariance of the Cauchy and related distributions}}
\author{Wooyoung Chin, Paul Jung, and Greg Markowsky}
\maketitle

	\begin{abstract}
	{It is known that if $f$ is an analytic self map of the complex upper half-plane which also maps $\mathbb{R}\cup\{\infty\}$ to itself, and $f(i)=i$, then $f$ preserves the Cauchy distribution. This note concerns three results related to the above fact.
	}

\end{abstract}


\tableofcontents





\vspace{2mm}

In \cite{pitman1967cauchy}, it was noticed that if $X$ has a standard Cauchy distribution and the finite or infinite sum 
\begin{align}\label{PW eq}
f(z)=az-\frac{b}z+\sum_{n} \frac{b_n}{a_n-z}-\frac{a_nb_n}{a_n^2+1}
\end{align}
 is a meromorphic function with all real constants such that 
\begin{enumerate}
	\item $a_n\neq 0$,
	\item any nonzero  constants $a,b,b_n$ all have the same sign, and
	\item $\sum \frac{b_n}{a_n^2}<\infty$ with $|a_n|\to\infty$,
\end{enumerate}
 then $(a+b+\sum_n\frac{b_n}{a_n^2+1})^{-1}f(X)$ also has a standard Cauchy distribution. The proof was a direct calculation of the characteristic function of $f(X)$ using a contour integral. Later, \cite{letac1977functions} gave a full characterization of measurable functions mapping $\R$ to $\R$ which preserve the set of all Cauchy distributions. At the heart of these results is the following idea: if $f$ is an analytic self map of the upper half of the complex plane, which also maps the extended real line to itself, then $f$,  appropriately centered and scaled, should preserve the standard Cauchy distribution {due to the fact that the Cauchy density is the Poisson kernel in the upper half-plane. This connection can be realized via the conformal invariance of Brownian motion and the fact that the exit distribution of Brownian motion from the upper half-plane is Cauchy (this was pointed out in \cite{letac1977functions}).}

{This note is based on our investigation into these results, and contains three sections which are related, but relatively independent of each other. In the first, we use the optional stopping theorem to calculate exit distributions of planar Brownian motion. In the second, we show how a particular invariant map of the standard Cauchy distribution, known as the Boole transformation\footnote{The name of this transformation is taken from \cite{adler1973ergodic}.}, and the ergodic theorem can be combined to give a quick proof of an interesting fact addressed in \cite{davarREU} and pointed out to us by D. Khoshnevisan}: when one tries in vain to use Newton's method to find real roots of $x^2+1$, the empirical distribution of the iterates converges to a Cauchy distribution. In the third, we show how recent ideas from  \cite{gross} can be used together with conformal invariance in order to find invariant maps for more general distributions. We then use this technique to provide an example of a map which leaves the hyperbolic secant distribution invariant.


\section{Optional stopping and exit distributions of planar Brownian motion} \label{sec:OSTexit}



In this section, we use the optional stopping theorem in order to calculate various exit distributions of planar Brownian motion via their characteristic or moment generating functions. The results in this section are all well-known and have been previously proved by other methods; however, our method produces {short and intuitive proofs, and we hope that the method may be found to be useful in other settings as well. To the best of our knowledge, this technique for proving our results is not present in the literature, even though it is a classical technique in the context of one dimensional Brownian motion.}


\begin{proposition}[Cauchy law for Brownian exit of the half-plane]  \label{thm:bm_exit_cauchy}
	Let $(B_t)_{t \ge 0}$ be a complex-valued Brownian motion starting at $i$.
	If $T$ is the time when $(B_t)_{t \ge 0}$ exits the upper half plane, then the distribution of $B_T$ is the standard Cauchy distribution.
\end{proposition}

\begin{proof}
	We use the optional stopping theorem to compute the characteristic function of $B_T$.
	Let $\theta \ge 0$ be given. Since $z \mapsto e^{i\theta z}$ is holomorphic, $M_t=(e^{i\theta B_t})_{t \ge 0}$ is a (complex-valued) martingale, and it is bounded on the upper half-plane since there $|M_t| = e^{-\theta Im(B_t)} \leq 1$. 
	Thus, by the optional stopping theorem,
	$$
		\E [e^{i\theta B_T}] = \E [e^{i\theta B_0}] = \E [e^{i\theta i}] = e^{-\theta}.
	$$
This argument fails when $\theta < 0$, because in that case $M_t$ is unbounded, however symmetry allows us to conclude easily that $\E [e^{i\theta B_T}]=e^{-|\theta|}$ for all real $\theta$, which is the characteristic function of the standard Cauchy distribution.
\end{proof}

{	\begin{remark}\label{remark}
	One immediately obtains the result regarding \eqref{PW eq} in the introduction. Suppose $B_0=i$ and $T$ is the first time $(B_t)_t$ hits $\R$.
	L\'evy's theorem on the conformal invariance of Brownian motion says that if $f$ is analytic, then $(f(B_t))_t$ is a time-changed Brownian motion (see \cite{dur} or \cite{bass}).
	 If in addition, $f$ maps the complex upper half-plane to itself and maps $\mathbb{R}\cup\{\infty\}$ to itself, then $f(B_T)$ is precisely the place where the time-changed Brownian motion first hits $\R$. Therefore, if $f(B_0)= f(i)=i$, then  $f(B_T)$ has a standard Cauchy distribution.
 \end{remark}}


\begin{proposition}[Sech law for Brownian exit of the strip] \label{thm:bm_exit_sech}
	Let $(B_t)_{t \ge 0}$ be a complex-valued Brownian motion starting at $0$.
	If $T$ is the time when $(B_t)_{t \ge 0}$ exits the infinite strip $\{-1 < Re(z) < 1\}$, then the distribution of $Im(B_T)$ has a hyperbolic secant law, characterized by density $\frac{\mbox{sech}(\frac{\pi}{2}y)}{2}dy$.
\end{proposition}

\begin{proof}
	We again use the optional stopping theorem to compute the characteristic function of $B_T$. Let $\theta \ge 0$ be given. Since $z \mapsto e^{\theta z}$ is holomorphic, $M_t=(e^{\theta B_t})_{t \ge 0}$ is a martingale, and it is bounded on $\{-1 < Re(z) < 1\}$ since $|M_t| = e^{\theta Re(B_t)} \leq e^{|\theta|}$. 
	Thus, by the optional stopping theorem,
	$$
		1=\E [e^{\theta B_0}] = \E [e^{\theta B_T}] = \E [e^{\theta Re(B_T)} e^{i\theta Im (B_T)}].
	$$

By symmetry, $\E [e^{\theta Re(B_T)}] = \frac{e^\theta + e^{-\theta}}{2} = \cosh \theta$, and furthermore the random variables $Re(B_T)$ and $Im(B_T)$ are independent so that the final expectation factors. From this we obtain $\E [e^{i\theta Im (B_T)}] = \mbox{sech} (\theta)$, and inverting the Fourier transform gives the result.

\end{proof}

We note that both of the preceding results can be obtained using the conformal invariance of Brownian motion and harmonic measure (see \cite{Markowsky18}), and Theorem \ref{thm:bm_exit_cauchy} can also be deduced by a direct calculation, using the Poisson kernel, or by properties of stable distributions; see for instance \cite[Sec. 1.9]{dur} or \cite[Ch. VI.2]{feller2008introduction}. To conclude this section, we give a new proof of a result from \cite{bour}, which was proved there by a different argument which also involved planar Brownian motion.

\begin{proposition} \label{}
If $C$ is a standard Cauchy random variable, then $$\E[e^{i \lambda \frac{2}{\pi} \ln |C|}] = \mbox{sech} \lambda.$$
\end{proposition}

\begin{proof}

The function $z \mapsto e^{i\lambda \frac{2}{\pi} \mbox{Log}(z)}$ is holomorphic and bounded on $\{Re(z) > 0\}$, where $\mbox{Log}(z) = \ln |z| + i Arg(z)$ and $Arg$ is the principal branch of the argument function, taking values in $(-\pi, \pi)$ on $\{Re(z) > 0\}$. Thus, if $(B_t)_{t \ge 0}$ is a complex-valued Brownian motion starting at $1$ and $T$ is the time at which $B_t$ exits $\{Re(z) > 0\}$, then $e^{i\lambda \frac{2}{\pi} \mbox{Log}(B_t)}$ is a bounded martingale for $0 \leq t \leq T$ and the optional stopping theorem gives
$$
1 = \E[e^{i\lambda \frac{2}{\pi} \mbox{Log}(B_0)}]=\E[e^{i\lambda \frac{2}{\pi} \mbox{Log}(B_T)}] =\E[e^{i\lambda \frac{2}{\pi} (\ln|B_T|+ i Arg(B_T))}].
$$

Now, by symmetry, the random variables $\ln|B_T|$ and $Arg(B_T)$ are independent, and $B_T$ is standard Cauchy (Theorem \ref{thm:bm_exit_cauchy}), while $Arg(B_T)$ is uniform on $\{-\frac{\pi}{2}, \frac{\pi}{2}\}$, so we obtain
$$
1= E[e^{i\lambda \frac{2}{\pi} \ln|C|}]\Big(\frac{e^\lambda + e^{-\lambda}}{2}\Big),
$$
and the result follows.

\end{proof}

\section{Trying to find an imaginary root using Newton's method}\label{sec:newton}

In this subsection we show that when one tries to use Newton's method to find real roots of $x^2+1$, one will asymptotically end up with a Cauchy distribution.
Recall that Newton's method finds real roots of a differentiable $f$ finds by using a sequence of approximations which are trying to successively get closer to a root.
Starting with an arbitrary initial guess $x_0$, Newton's method computes the next approximation using
$$
	x_{n+1} = x - \frac{f(x_n)}{f'(x_n)},
$$
while hoping that we would never have $f'(x_n) = 0$.


For $x^2+1$, Newton's method computes the approximations using
$$
	x_{n+1} = \varphi(x_n),
$$
where $\varphi: \R \setminus \set{0} \to \R$ is given by the Boole transformation
$$
	\varphi(x) = x - \frac{x^2+1}{(x^2+1)'} = \frac{1}{2}\left(x - \frac{1}{x}\right).
$$
If we let {$$N \coloneqq \set{x \in \R}{\text{$\varphi^n(x)=\infty$ for some $n \in \N$}},$$} then $N$ is countable, and $\varphi$ maps the extended real line into itself.
Restrict the domain and codomain of $\varphi$ to $\R \setminus N$.

\begin{proposition}[Newton's method applied to $x^2+1$] \label{thm:newton_cauchy}
	For Lebesgue-almost every $x$ and for any measurable function $f: \R \to \R$ such that $f(x)/(1+x^2)$ is integrable,
	we have
	$$
		\frac{1}{n}\sum_{k=0}^{n-1} f(\varphi^{k}(x)) \to \int \frac{f(t)}{\pi(1 + t^2)} \,dt.
	$$
	In particular, for Lebesgue-almost every $x$, we have
	$$
		\frac{1}{n} \sum_{k=0}^{n-1} \delta_{\varphi^k(x)} \wto \frac{1}{\pi(1+t^2)}dt.
	$$
\end{proposition}

\noindent {\bf Remark.} Newton's method has several variants such as a trapezoidal-Newton's method or Simpson-Newton's method, to which the above proposition also applies. One simply replaces $\varphi$ by
$$
	\psi(x) \coloneqq \frac{x^3 - 3x}{3x^2 - 1},\qquad x\in\R \setminus \set{1/\sqrt{3}, -1/\sqrt{3}} .
$$
All the arguments in the proof of the proposition can then be extended to the map $\psi$.

In order to prove the proposition, we make use of Birkhoff's ergodic theorem and the ergodicity of $\varphi$. The ergodicity of $\varphi$ is well-known, but for completeness we include an argument here. We first consider a related map on $S^1$, namely $z \mapsto -z^2$, and show that this map on $S^1$ is ergodic.

\begin{lemma}[Ergodicity of the doubling map] \label{lem:minus_double_ergodic}
	Assume that $S^1$ is equipped with the usual measure, whose total measure is $2\pi$. The map $\rho:S^1 \to S^1$ given by $\rho(z) \coloneqq -z^2$ is ergodic.
\end{lemma}

\begin{proof}
	Assume that $A \subset S^1$ is a measurable set with $1_{\rho^{-1}(A)} = 1_A$ a.s.
	Let $1_A(z) = \sum_{n \in \Z} c_n z^n$ where the infinite summation is in the $L^2$ sense.
	It is not difficult to see that $$1_{\rho^{-1}(A)} = 1_A (-z^2) = \sum_{n \in \Z} c_n (-z^2)^n = \sum_{n \in \Z} (-1)^nc_n z^{2n},$$ again in the $L^2$ sense.
	Since $1_{\rho^{-1}(A)} = 1_A$ a.s., the coefficients of the two series must match, and so $c_{2n} = (-1)^nc_n$ for each $n \in \Z$.
	Since $|c_n|=|c_{2n}|=|c_{4n}|=\cdots$ and $\lim_{n \to \infty} c_n = \lim_{n \to -\infty} c_n = 0$, we have $c_n = 0$ for all $n \ne 0$.
	Thus $1_A$ must be constant a.s., and this shows that $\rho$ is ergodic.
\end{proof}

%

Using Lemma \ref{lem:minus_double_ergodic}, we now prove that $\varphi$ is ergodic.
Let $\C_+$ be the upper half plane, and $\mathbb{D}$ the unit open disk.
Let $F:\C_+ \cup \R \to \overline{\mathbb{D}} \setminus \set{-1}$ and $G: \overline{\mathbb{D}} \setminus \set{-1} \to \C_+ \cup \R$ be defined by
$$
	F(x) \coloneqq \frac{i-x}{i+x}
$$
and
$$
	G(x) \coloneqq i\frac{1-z}{1+z}.
$$
Then $F$ and $G$ are inverses to each other.

\begin{lemma}[Ergodicity of the Boole transformation] \label{lem:newton_ergodic}
	With respect to the standard Cauchy distribution, the Boole transformation,
	$$
		\varphi(x) = \frac{1}{2}\left(x - \frac{1}{x}\right),\qquad x\in \R \setminus N,
	$$
	is ergodic.
\end{lemma}

\begin{proof}
	Let $D \coloneqq F(\R \setminus N)$. Restrict the domains and codomains of $F$ and $G$ by $F:\R \setminus N \to D$ and $G:D \to \R \setminus N$.
	For each $z \in D$, we have
	\begin{align*}
		(F \circ \varphi \circ G)(z) &= F\left( \frac{1}{2} \left( i\frac{1-z}{1+z} + i\frac{1+z}{1-z}  \right) \right)
		= F\left( \frac{i}{2} \cdot \frac{(1-z)^2 + (1+z)^2}{1-z^2} \right) \\
		&= F \left( i \frac{1+z^2}{1-z^2} \right)
		= \frac{1 - \frac{1+z^2}{1-z^2}}{1 + \frac{1+z^2}{1-z^2}}
		= \frac{-2z^2}{2} = -z^2.
	\end{align*}
		
	Assume that $A \subset \R \setminus N$ is a measurable set with $1_{\varphi^{-1}(A)} = 1_A$ a.s.
	Since $F$ and $G$ map measure-zero sets to measure-zero sets, we have
		$$1_{F(A)} = 1_{F(\varphi^{-1}(A))} = 1_{(F \circ \varphi \circ G)^{-1}(F(A))} \quad \text{a.s.}$$
	By Lemma \ref{lem:minus_double_ergodic}, the measure of $F(A)$ should be either $0$ or $1$, and therefore the measure of $A$ is either $0$ or $1$. This shows that $\varphi$ is ergodic.
\end{proof}

%

\begin{proof}[Proof of Proposition \ref{thm:newton_cauchy}]
	Note that $\varphi(i) = i$. Also, if $z \in \C_+$, then $1/z \in -\C_+$, and thus $-1/z \in \C_+$.
	Therefore,
	$$	\varphi(z) = \frac{1}{2}\left(z - \frac{1}{z}\right) \in \C_+$$
	and $\varphi$  maps $\C_+$ into itself.
	{Thus, 
	by \eqref{PW eq} (see also Remark \ref{remark})} we have that $\varphi$ preserves the standard Cauchy distribution.
	
	With respect to the standard Cauchy distribution, the map $\varphi$ is measure-preserving and ergodic, thus the desired conclusion follows from the ergodic theorem (see for instance, \cite[Theorem 7.2.1]{DurrettPTE}).
\end{proof}

\section{Invariance of the hyperbolic secant distribution}\label{sec:sech}

Invariant maps of the Cauchy distribution can also be used to find functions that preserve distributions other than the Cauchy, for instance \cite{PitmanYor04} applies this theorem to find a family of rational maps under which the arc-sine law is invariant. We can adapt this technique to more general distributions, as follows. Suppose that $W$ is a simply connected domain in $\CC$ which is symmetric about $\RR$; that is, $z \in W$ if, and only if, $\bar z \in W$. Suppose further that the boundary components of $W$ in the upper and lower half-planes can each be viewed as the graph of a continuous function with dependent variable $y$ and independent variable $x$; equivalently, any vertical line $z(t)=x+it$ intersects $\partial W$ at exactly two points, which are necessarily conjugates of each other. Let 
\begin{align}\label{tau}
\tau = \inf\{t \geq 0: B_t \in \partial W\},\end{align} 
and let $\DD_a$, for $a \in \RR$, denote the distribution of $Re(B_\tau)$ under the condition $B_0 = a$ a.s. It was shown in a recent elegant paper \cite{gross}, that any distribution satisfying certain moment conditions can be realized by a simply connected domain in this manner. For $x \in \RR$, let $\pi(x)$ be the unique point $z$ in the upper half-plane with $z \in \partial W$ and $Re(z) = x$. We have the following simple proposition.

\begin{proposition} \label{zoe}
If $X \sim \DD_a$ for some $a \in \RR$ and $f$ is a conformal automorphism of $W$ {which maps $\RR$ into itself}, then $Re(f(\pi(X))) \sim \DD_{f(a)}$.
\end{proposition}

\begin{proof}
The conditions on the domain imply that $W$ is conformally equivalent to a Jordan domain by a M\"obius transformation. By Carath\'eodory's theorem (see \cite{goluzin}), $f$ extends to a continuous bijection from the closure of $W$ (in the Riemann sphere) to itself, thus $f$ is defined on $\partial W$. The fact that $f(W \cap \RR) \subseteq \RR$ implies that the analytic functions $f(z)$ and $\seg{15}{f(\bar z)}$ agree on $\RR$, and therefore by the uniqueness principle for analytic functions (see \cite{rud}) we have $f(z) = \seg{15}{f(\bar z)}$ and also $f(\bar z)=\seg{15}{f(z)}$, for all $z \in (W \cup \partial W)$. Now for $\tau$ as in \eqref{tau}, 
$$X \sim \DD_a \sim Re(B_\tau),$$ and by L\'evy's theorem on the conformal invariance of Brownian motion, we have $$\DD_{f(a)} \sim Re(f(B_\tau)) = Re(f(Re(B_\tau) + i Im(B_\tau))).$$ However, $f(\bar z)=\seg{15}{f(z)}$ implies $$Re(f(Re(B_\tau) + i Im(B_\tau))) = Re(f(Re(B_\tau) + i |Im(B_\tau)|)) = Re(f(\pi(Re(B_\tau)))).$$ The result follows upon replacing $Re(B_\tau)$ with $X$.
\end{proof}

Naturally, the difficulties in applying this result are $(a)$ identifying the distribution of $Re(B_\tau)$, and $(b)$ finding a conformal automorphism of the required type. We can, however, give an example as follows.

Let $W = \{-\frac{\pi}{2}<Im(z)< \frac{\pi}{2}\}$. The function $\psi(z) = ie^{z}$ maps $W$ conformally onto the upper half-plane, so using the function $\varphi(z) = \frac{1}{2}\left(z - \frac{1}{z}\right)$ from Section \ref{sec:newton} we see that $f=\psi^{-1} \circ \varphi \circ \psi$ is a conformal self-map of $W$ of the type required for Proposition \ref{zoe}. Calculating, we have

$$
f(z) = \mbox{Log}\(- \frac{i}{2} (ie^z - \frac{1}{ie^z})\) = \mbox{Log} \(\frac{e^z + e^{-z}}{2}\),
$$
where $\mbox{Log}$ denotes the principal branch of the logarithm. If $g(x) = Re(f(\pi(x)))$ then

$$
g(x) = \mbox{Log} \lb \frac{e^{x+i\pi/2} + e^{-x-i\pi/2}}{2}\rb = \mbox{Log} \lb\frac{e^x - e^{-x}}{2}\rb = \mbox{Log} |\sinh x|.
$$

Using the exit distribution of $(B_t)_t$ from $W$ derived in Proposition \ref{thm:bm_exit_sech}, we obtain the following result.

\begin{corollary}[An invariant map for the $\sech$ law] The distribution $\frac{1}{2} \sech (\frac{\pi}{2}x)\, dx$ is invariant under the transformation
	\begin{align*}
		g(x) := \frac{2}{\pi} \mbox{Log} \lb \sinh \(\frac{\pi}{2}x\)\rb.
	\end{align*}
\end{corollary}

\section*{Acknowledgements}   W. Chin and P. Jung are supported in part by (South Korean) National Research Foundation grant N01170220. This project began while G. Markowsky was visiting KAIST, and he would like to thank the mathematics department there for their kind hospitality. We also thank Davar Khoshnevisan and Edson de Faria for helpful conversations.

\bibliographystyle{alpha}
\bibliography{references}
\end{document}